\documentclass[12pt]{article}       
\usepackage{graphicx}
\usepackage{calc}
\usepackage{array}
\usepackage{graphicx}
\usepackage{subfigure}
\usepackage[indent,bf]{caption}
\usepackage{rotating}
\usepackage{setspace}
\usepackage{longtable}
\usepackage{mdframed}
\usepackage{rotating}
\usepackage{hyperref}

\usepackage{amssymb, amsmath, amsfonts, amsthm, epsfig}
\usepackage{epsfig,latexsym}
\usepackage{multicol, paralist}
\usepackage{pgf,tikz}
\usetikzlibrary{arrows,snakes,backgrounds, automata}
\usepackage{verbatim}
\usepackage{url}
\usepackage{fancyhdr}
\usepackage{authblk}
\usepackage{marginnote}

\usepackage{geometry}

\usetikzlibrary{arrows}

\newtheorem{thm}{Theorem}

\newtheorem{conj}[thm]{Conjecture}
\newtheorem{ques}{Question}
\newtheorem{lem}[thm]{Lemma}
\newtheorem{prop}[thm]{Proposition}
\newtheorem{cor}[thm]{Corollary}

\newtheorem*{Meyniel}{Meyniel's Conjecture}

\DeclareMathOperator{\ecop}{\overline{c}}
\DeclareMathOperator{\cop}{c}
\DeclareMathOperator{\capt}{capt}

\DeclareMathOperator{\diam}{diam}
\DeclareMathOperator{\Om}{\Omega}
\DeclareMathOperator{\Th}{\Theta}

\begin{document}

\title{A study of cops and robbers in oriented graphs
}

\author[1]{Devvrit Khatri}
\affil[1]{Department of Computer Science and Engineering, Birla Institute of Technology and Science, Pilani, India}
\author[2]{Natasha Komarov}
\affil[2]{Department of Mathematics, Computer Science, and Statistics, St. Lawrence University, Canton, NY, USA}
\author[3]{Aaron Krim-Yee}
\affil[3]{Department of Bioengineering, McGill University, Montreal, QC, Canada}
\author[4]{Nithish Kumar}
\affil{Department of Computer Science and Engineering, National Institute of Technology -- Tiruchirapalli, India}
\author[5,6]{Ben Seamone}
\affil[5]{Mathematics Department, Dawson College, Montreal, QC, Canada}
\affil[6]{D\'{e}partement d'informatique et de recherche op\'{e}rationnelle, Universit\'{e} de Montr\'{e}al, Montreal, QC, Canada}
\author[7]{Virg\'elot Virgile}
\affil[7]{D\'epartement de math\'ematiques et de statistique, Universit\'{e} de Montr\'{e}al, Montreal, QC, Canada}
\author[5]{AnQi Xu}




\renewcommand{\thefootnote}{\roman{footnote}}	

\maketitle
\renewcommand{\thefootnote}{\arabic{footnote}}	

\begin{abstract}
We consider the well-studied cops and robbers game in the context of oriented graphs, which has received surprisingly little attention to date.  We examine the relationship between the cop numbers of an oriented graph and its underlying undirected graph, giving a surprising result that there exists at least one graph $G$ for which every strongly connected orientation of $G$ has cop number strictly less than that of $G$.  We also refute a conjecture on the structure of cop-win digraphs, study orientations of outerplanar graphs, and study the cop number of line digraphs.  Finally, we consider some the aspects of optimal play, in particular the capture time of cop-win digraphs and properties of the relative positions of the cop(s) and robber.
\end{abstract}

\section{Introduction}

The game of Cops and Robbers played on graphs was introduced independently by Quillot \cite{Q83} and Nowakowski and Winkler \cite{NW83}.  The game is played between a set of pursuers (cops) and an evader (robber) who move from vertex to vertex in a graph.  The standard rules of play are that the cops initially chose their positions, followed by the robber choosing his.  Each cop may then move to an adjacent vertex or stay in place.  After the cops take their turns, the robber may move to an adjacent vertex or stay in place.  This constitutes one round of play.  The cops win if at least one cop is able to occupy the same vertex as the robber after a finite number of rounds; the robber wins if he can avoid capture indefinitely.

There are two standard parameters of interest in the standard cops and robbers model:
	\begin{compactitem}
	\item $\cop(G)$, or the \textit{cop number} of a graph $G$, is the fewest number of cops which guarantee that the cops can win the game;
    \item $\capt(G)$, or the \textit{capture time} of $G$, is the fewest number of rounds in which $\cop(G)$ cops can win.
    \end{compactitem}

While a significant amount of work has been done in studying cops and robbers on undirected graphs (see \cite{BN} for a thorough treatment of all topics related to cops and robbers), comparatively little is known about cops and robbers on directed graphs.
A general directed graph may have loops or parallel arcs.  An \textit{oriented graph}, or ograph, is a directed graph with no loops and no parallel arcs; in other words it is a simple graph whose edges have each been assigned some orientation.  Ographs will be the primary focus of this work\footnote{Since we allow each agent to stay on their current vertex on their turn, there is nothing lost or gained by imagining a loop present on every vertex.}.

Let us begin by surveying some of the results known regarding cops and robbers on directed graphs.  One of the first papers to address the problem is \cite{AF84}, where a simple forbidden subgraph condition is given to ensure that $\cop(D) \geq \delta^+(D)$.  Directed graphs are further considered in \cite{H87}, where upper bounds are provided for directed abelian Cayley graphs.  In \cite{FKL12}, it is shown that every strongly connected digraph $D$ satisfies $\cop(D) \in O(n\tfrac{(\log\log{n})^2}{\log{n}})$.  Vertex transitive oriented grids and the cop number of their quotients are studied in \cite{Hthesis,HM18}, and orientations of planar graphs are considered in \cite{LO17}.  It is shown in \cite{Hthesis} that there exist ographs with $\Delta^+$ and $\Delta^-$ each at most $2$ having unbounded cop number; the minimum order of a $3$-cop-win digraph is shown to be either $7$ or $8$.
Hahn and MacGillivray \cite{HM06} give an algorithmic characterization of cop-win digraphs, as well as a method for reducing the game played with $k$ cops and $l$ robbers to the $1$ cop, $1$ robber case.  Unfortunately, these results do not give a structural characteristic of such graphs (we address this problem in Section \ref{killtheconj}).  Kinnersley \cite{K15} has shown that determining the cop number of a graph or digraph is EXPTIME-complete.  The capture time of directed graphs has also been examined by Kinnersley \cite{K18}, who showed that for every $k\geq 1$ there exists an $n$-vertex directed graph for which $\cop(D) =k$ and $\capt(D) \in \Th((\tfrac{n}{k})^{k+1})$.  It should finally be noted that work has been done on a version of cops and robbers in undirected and directed graphs where the robber can move infinitely fast (among many other variations).  The number of cops needed in this setting is often tied to various width parameters of the graph.  The literature is extensive, and since it is not the focus of this work we will not survey those results here.

This paper is structured as follows.  In Section \ref{gen}, we survey some basic results about the cop number of an oriented graph.  In particular, we examine its relationship to the cop number of the underlying undirected graph, showing that $\cop(D)$ and $\cop(G)$ are incomparable in general, where $G$ is the underlying graph of an ograph $G$, even in the case when $D$ is strongly connected.  Section \ref{killtheconj} is devoted to disproving a conjecture posed in \cite{DGGH16} which attempted to characterize all cop-win ographs.  In Section \ref{planar}, we prove that all strongly connected orientations of outerplanar graphs have bounded cop number; a similar result for general planar graphs remains an open problem.  Section \ref{linedigraphs} explores the relationship between the cop number of an ograph and the cop number of its arc digraph.  Finally, in Section \ref{strategy}, we consider some problems related to the optimal play of the agents.  We give a construction of an $n$-vertex cop-win ograph for which the capture time is in $\Th(n^2)$ (in contrast with the known result that all undirected cop win graphs have capture time at most linear in $n$).  We also consider whether or not a cop can be made to revisit a previously visited vertex or a cop can be made to increase her distance from the robber in optimal play.  These questions were answered positively in \cite{B+13} for undirected graphs, and we show that indeed both are possible in ographs as well.

\section{Behaviour of the cop number in ographs}\label{gen} 


One of the first results in cops and robbers is a structural characterization of those graphs which have cop number equal to $1$; such graphs are called \textit{cop-win}.  As mentioned in the introduction, no such characterization is currently known for digraphs.  It is easy to see that if a digraph $D$ contains sources, then $\cop(D)$ must be at least as large as the number of sources, and so a cop-win digraph must have at most one source.  It is also not hard to see that a graph with no source must necessarily have cop number at least $2$.  To see this, note that if a single cop starts on vertex $v$, then the robber starts on some vertex $u$ for which $(u,v) \in A(D)$.  After each cop turn, the robber moves to the cop's previous position.  This observation that a cop-win graph must have only one source was further refined as follows:

\begin{prop}\cite{DGGH16}\label{onesource}
If $D$ contains exactly one source and no directed cycles, then $\cop(D) = 1$.
\end{prop}

We will return to the topic of cop-win digraphs in Section \ref{killtheconj}, but now turn our attention to the relationship between the cop number of a digraph $D$ and its underlying graph $G$.

We begin with a simple observation.

\begin{prop}\label{copwinorientation}
Every graph $G$ has an orientation $D$ such that $\cop(D) = 1$
\end{prop}

\begin{proof}
Let $v \in V(G)$, and let $X_i \subset V(G)$ be those vertices of $G$ at distance $i\geq 0$ from $v$ (so $X_0 = \{v\}$)\footnote{In other words, construct a BFS tree with $v$ as its root.}.  For every $xy \in E(G)$, either $x,y \in X_i$ or $x \in X_i, y \in X_{i+1}$ for some $i$.  For each $i$, orient the edges in $X_i$ according to some linear order, and direct any $X_iX_{i+1}$ edge from $X_i$ to $X_{i+1}$.  This graph has no oriented cycles and exactly one source, and so has cop number $1$ by Proposition \ref{onesource}.
\end{proof}

Knowing that any graph can be oriented so that the cop number is as small as possible, we now show that any graph can be oriented so that the cop number is, in some sense, large.

\begin{thm}\label{BFSorientation}
Every graph $G$ has an orientation $D$ for which $\cop(D) \geq \lceil \frac{\diam(G)}{2}\rceil$.
\end{thm}

\begin{proof}
Let $v_1,v_2,v_3,..,v_d$ be the vertices on a diametric path of the graph $G$.  Construct a BFS tree rooted at $v_1$; call $\{v_1\}$ level $0$ and call those vertices at distance $d$ from $v_1$ level $d$.  For every edge $e$ with ends in different levels, orient $e$ so that its head lies in an odd level; orient all other edges arbitrarily (call this an \textit{alternating BFS orientation}).  Clearly at least one cop is required in each even level, and so $\cop(D) \geq \lceil\frac{\diam(G)}{2}\rceil$.
\end{proof}

Note that there exist graphs for which $\lceil \frac{\diam(G)}{2}\rceil$ is the largest possible value for $\cop(D)$ taken over all orientations of $D$ (paths, for example).
We may refine the bound in Theorem \ref{BFSorientation} by constructing a BFS tree as in the proof of Theorem \ref{BFSorientation}, and giving is an alternating BFS orientation where each sub-digraph induced by a layer has compenents which are also alternating BFS orientations.

\begin{thm}\label{BFSorientation2}
For any $G$, there exists an alternating BFS orientation $D$ for which $\cop(D) \geq \sum_{i=0}^{\lceil \diam(G)/2 \rceil}\lceil \frac{\diam(G_{2i})}{2}\rceil$, where $G_j$ denotes the subgraph induced by the vertices at level $j$ in the alternating BFS orientation of $G$.
\end{thm}

The natural question to ask is whether or not there is any relationship between $\cop(D)$ and $\cop(G)$ where $G$ is the underlying graph of $D$.  We clearly see that this is not the case if $D$ is not required to be strongly connected.  Proposition \ref{copwinorientation} guarantees that a graph with arbitrarily high cop number may be oriented so that it then has cop number $1$, and so $\cop(D) < \cop(G)$ in this case.  On the other hand, a cop-win graph with large diameter (for instance, a path) can be oriented so that the cop number becomes arbitrarily large, or $\cop(D) > \cop(G) = 1$.  It is not difficult to see that this inequality can hold for general graphs.

\begin{prop}\label{goesup}
For every graph $G$, there exists an orientation $D$ of $G$ such that $\cop(D) \geq \cop(G)$.
\end{prop}

\begin{proof}
It is well-known and easy to see that $\cop(G) \leq \gamma(G) \leq \alpha(G)$, where $\gamma(G)$ denotes the domination number of $G$ and $\alpha(G)$ denotes the independence number of $G$.  If $X$ is a maximum independent set of $G$, then orient the edges so that no head of an arc lies in $X$; call this orientation $D$.  Clearly one requires at least $|X|$ cops to win on $D$, since otherwise at least one source will be uncovered at the initial step.
\end{proof}

Clearly, this construction relies on using a large number of sinks to ensure that a large number of cops cannot fail from the get-go.  What if we look at orientations with no sinks?  

A natural place to turn are orientations of complete graphs, or \textit{tournaments}.  The PhD thesis of Hill \cite{H08} presents a small handful of results on the cop number of tournaments.  For example, it is shown that a tournament is cop-win if and only if it has a dominating vertex; this result is re-presented in \cite{B+13} along with an analysis of tournaments with very large maximum degree ($\Delta = |V(T)| - c$ for a fixed constant $c$).  Hill also offered a tempting conjecture that there is some absolute constant $k$ such that $\cop(T) \leq k$ for any tournament $T$.  A particular case of this conjecture was communicated by Nowakowski to the authors of \cite{B+13}, positing that tournaments obtained from a Steiner triple system on $n$ vertices by cyclically orienting each triple arbitrarily (called a Steiner triple graph), have cop number at most $2$.  Both conjectures were disproven by Sl\'ivlov\'a in \cite{S15}, who showed that there exists a Steiner triple graph $T$ which $\cop(T) > k$ for every $k \in \mathbb{N}$.  In our preparation of this paper, we noted a few small corrections to be made in the proof given in \cite{S15}; a full proof is given here for completeness.  

\begin{thm}\cite{S15}\label{STS}
For every $k \in \mathbb{N}$, there exists an oriented Steiner triple graph $T$ for which $\cop(T) \geq k$.
\end{thm}

Let $S_n$ be the set of oriented Steiner triple systems on $n$ vertices. Let $T \in S_n$ have adjacency matrix $A$. Finally, let $B = A + I$. We will say that $T \in T_n \subseteq S_n$ if $T$ satisfies the following property:
\begin{equation}
\label{slivova property}
 \forall  r, c_1, \dots, c_k \in V(T), \exists s \in V(T), \mbox{ such that } B_{r,s}=0 \mbox{ and } B_{c_i,s}=1 \mbox{ for all } i\in [k].    
\end{equation}
Intuitively: if the robber is on vertex $r$ and there are $k$ cops on vertices $c_1, c_2, \dots, c_k$, then $s$ is a ``safe'' vertex if the robber can move there while none of the $k$ cops can do so on their next turn. The set $T_n$ consists of those oriented Steiner triple systems $T$ in which the robber always has a safe move to make, and therefore cannot be captured by $k$ cops. Our goal is to show that $T_n \neq \emptyset$. 

\begin{proof}
Fix $k \in \mathbb{N}$. For any $n \in \mathbb{N}$, let $p_n$ be the probability that an oriented Steiner triple system $T$, chosen uniformly at random from $S_n$, is not in $T_n$ --- that is, it fails to satisfy the conditions of property (\ref{slivova property}). 
We shall show $\displaystyle \lim_{n \rightarrow \infty} p_n = 0$, proving that there is an oriented Steiner triple system $T$ with $\cop(T) > k$.

Suppose that the robber is at a vertex $r$ and the $k$ cops are at vertices $c_1, c_2, \dots, c_k$, and let $q$ be the probability that no matter where the robber moves from $r$, he will be captured by one of the cops. 
That is, $q$ is the probability that the following condition is satisfied: 
\begin{equation}
    \label{capture condition}
    \forall v \in V(T), B_{r,v} {=} 1 \implies B_{c_i,v}{=}1 \mbox{ for some } i\in[k].
\end{equation}

We will only consider those vertices $v$ such that $B_{c_i,v}$ and $B_{c_j,v'}$ are independent whenever $i\neq j$ or $v\neq v'$. Sl\'ivlov\'a \cite{S15} uses Tur\'{a}n's theorem to show that there are at least $\frac{n-k-{k \choose 2}}{4k+1}$ such vertices. In the rest of this proof, we will consider only those vertices of $V(T)$ satisfying this condition.

For a given vertex $v$, the probability that condition (\ref{capture condition}) is met is $1-\frac{1}{2^k}$,
and therefore

$$q \le \left(1-\frac{1}{2^k} \right)^\ell \le \left(1-\frac{1}{2^k}\right)^\frac{n-k-{k \choose 2}}{4k+1}.$$

Let $I \subset \mathbb{N}$ be an index set for the subsets $\{c_1, c_2, \dots, c_k\}$ of $\displaystyle {V(T) \choose k}$, and for all $i\in I$ let $X_i$ be an indicator random variable which is 1 if subset $i$ fails to satisfy condition~(\ref{slivova property}) and 0 otherwise. There are $n$ ways to choose a robber location $r$ and $k^{n{-}1}$ ways to choose the $k$ cop positions $c_1, c_2, \dots, c_k$; for each choice, $\mathbb{P}(X_i) \le q$. Let $X$ be the number of subsets failing to satisfy condition~(\ref{slivova property}).
Therefore we have

\begin{eqnarray*}
    \lim_{n \rightarrow \infty} \mathbb{E}[X]   &\le&  \lim_{n \rightarrow \infty}n k^{n-1} q \\
            &\le& \lim_{n \rightarrow \infty} n (n-1)^k \left(1-\frac{1}{2^k}\right)^\frac{n-k-{k \choose 2}}{4k+1} \\
            &=& 0
\end{eqnarray*}

which tells us that $\displaystyle \lim_{n \rightarrow \infty} p_n = 0$, as desired.
\end{proof}

\begin{cor}
For any $k$, there exists a strongly oriented digraph $D$ such that $\cop(D) - \cop(G) > k$ where $G$ is the underlying graph of $D$.
\end{cor}

It is natural to ask, then, is $\cop(D)$ always greater than $\cop(G)$ when $D$ is strongly connected?
Perhaps, surprisingly, the answer is no.  Hosseini shows in \cite{Hthesis} that for any $k \geq 2$ there exists a graph $G$ and a strong orientation $D$ of $G$ for which $\cop(G) = k$ and $\cop(D) = 2$.  This construction involves using very long directed paths so that if the robber enters the path he has no hope of escaping capture.  We complement this result by showing that there exists a graph for which $\cop(D) < \cop(G)$ for \textit{any} orientation $D$ of $G$.

\begin{thm}
If $D$ is a strongly connected orientation of the Petersen graph, then $\cop(D) = 2$.
\end{thm}

The proof of this theorem is by computer search.  A polynomial time algorithm to check if the cop number of a graph was first presented in \cite{BI93}; a more general version which applies to graphs and directed graphs was then given in \cite{HM06}.  Our result follows from an implementation of a version of the algorithm presented in \cite{BC09} which can also be found in \cite{BN}, applied to the 1920 different strong orientations of the Petersen graph.  Since the cop number of the (undirected) Petersen graph is $3$, we obtain the following corollary:

\begin{cor}\label{cangodown}
There exists a graph $G$ for which $\cop(D) < \cop(G)$ for every strong orientation $D$ of $G$.
\end{cor}

A treatment of the cop number parameter would not be complete without a mention of Meyniel's Conjecture:

\begin{Meyniel}
For every undirected graph $G$, $\cop(G) \in O(\sqrt{|V(G)|})$.
\end{Meyniel}

The upper bound of $c\sqrt{|V(G)|}$ is typically referred to as the \textit{Meyniel bound} for the cop number.  It is trivial to see that the Meyniel bound cannot hold for general digraphs, since a digraph may have as many as $n-1$ sources.  However, it is unknown whether or not it holds for strongly connected digraphs in general, though it is known that it holds for digraphs of diameter 2 and bipartite digraphs of diameter 3 \cite{Hthesis}.  In light of Corollary \ref{cangodown}, we briefly confirm that one cannot hope for better than the Meyniel bound in this case.
The following lemma, which is an adaptation of a similar result on undirected graphs from \cite{AF84}, is needed:

\begin{lem}\cite{LO17}\label{dirlowbound}
If $D$ is an oriented graph with undirected girth at least $5$, then $\cop(D) \geq \delta^+(G)$.
\end{lem}

\begin{thm}\label{Meyniellower}
For any $N \in \mathbb{N}$, there exists an $n$-vertex strongly connected ographs with cop number in $\Om(\sqrt{n})$ and $n \geq N$.
\end{thm}

\begin{proof}
Let $G$ be the point-line incidence graph of the projective plane of order $q$ with partite sets $\mathcal{P}$ and $\mathcal{L}$.
It is known that $G$ is Hamiltonian; orient the edges of some Hamiltonian cycle $C$ cyclically.  What remains is a $(q-1)$-regular bipartite graph, and so it can be decomposed into $q-1$ perfect matchings.  Orient $\lceil \frac{q-1}{2} \rceil$ of the matchings from $\mathcal{P}$ to $\mathcal{L}$ and the remaining $\lfloor \frac{q-1}{2} \rfloor$ matchings from $\mathcal{L}$ to $\mathcal{P}$.  The resulting oriented graph $D$ is strongly connected, has $n = 2q^2+2q+2$ vertices, has undirected girth $6$, and has $\delta^+ = \lfloor \frac{q+1}{2}\rfloor$.\end{proof}

Finally, we note that there can be no general relationship between the cop numbers of a directed graph in the usual setting and in the fully active setting (where no agent may stay put on his/her turn, see \cite{GKS18}), even when we insist on considering strongly connected digraphs.  To see this, we simply note that a directed $n$-cycle has cop number $2$ in the regular game and $\lceil \frac{n}{2} \rceil$ in the fully active setting.

\section{Cop-win ographs}\label{killtheconj}

In \cite{DGGH16}, an attempt to characterize cop-win ographs is undertaken.  Recall that Proposition \ref{onesource} states that an ograph with exactly one source and no directed cycles is cop-win.  The converse of this is false; orient $K_4$ so that one triangle is cyclically oriented and the remaining vertex is a source.  The authors in  \cite{DGGH16} define a directed cycle to be \textit{cop-dominated} if ``the robber cannot avoid capture by entering the directed cycle by some path and then simply traveling repeatedly on the directed cycle.''  They propose the following:

\begin{conj}\cite{DGGH16}\label{copwinconj}
An oriented graph $D$ is cop-win if and only if $D$ contains exactly one vertex of in-degree $0$, every vertex is reachable from that vertex, and every directed cycle of $G$ is cop-dominated.
\end{conj}

We disprove this by counterexample:

\begin{thm}
There exists an ograph $D$ with $\cop(D) = 2$, exactly one source, and every directed cycle being cop-dominated.
\end{thm}

\begin{proof}

Consider the graph in Figure \ref{counterexample}, and suppose that it is cop-win.  Since $a$ is its sole source, it is the unique feasible starting position for the cop.

\begin{figure}[h]
\begin{center}
\scalebox{0.68}{
\begin{tikzpicture}
\tikzset{vertex/.style = {shape=circle,draw,minimum size=1.2em}}
\tikzset{edge/.style = {->,> = latex'}}
\node[vertex] (a) at  (0,0) {$a$};

\node[vertex] (b) at  (2,3) {$b$};
\node[vertex] (g) at  (2,-3) {$g$}; 
\node[vertex] (c) at (7,5) {$c$};
\node[vertex] (e) at (7,-5) {$e$};
\node[vertex] (d) at (12,3) {$d$};
\node[vertex] (f) at (12,-3) {$f$};

\node[vertex] (h) at (3,0) {$h$};
\node[vertex] (i) at (5,0) {$i$};
\node[vertex] (j) at (7,0) {$j$};
\node[vertex] (k) at (9,0) {$k$};
\node[vertex] (l) at (11,0) {$l$};


\draw[edge] (a) to (b);
\draw[edge] (a) to (g);
\draw[edge] (a)  to[bend left=90] (c);
\draw[edge] (a)  to[bend left=90] (d);
\draw[edge] (a)  to[bend left=270] (e);
\draw[edge] (a)  to[bend left=270] (f);

\draw[edge] (h)  to[bend left] (j);
\draw[edge] (l)  to[bend left] (j);

\draw[edge] (g)  to[bend left=25] (c);
\draw[edge] (d)  to[bend left=25] (e);

\draw[edge] (b) to (g);
\draw[edge] (g) to (e);
\draw[edge] (e) to (f);
\draw[edge] (f) to (d);
\draw[edge] (d) to (c);
\draw[edge] (c) to (b);

\draw[edge] (g) to (h);
\draw[edge] (b) to (i);
\draw[edge] (e) to (i);
\draw[edge] (e) to (j);
\draw[edge] (f) to (k);
\draw[edge] (c) to (j);
\draw[edge] (c) to (k);
\draw[edge] (d) to (l);

\draw[edge] (j) to (i);
\draw[edge] (i) to (h);

\draw[edge] (j) to (k);
\draw[edge] (k) to (l);



\end{tikzpicture}
}
\end{center}
\caption{A counterexample to Conjecture \ref{copwinconj}}\label{counterexample}
\end{figure}

The only oriented cycles in the graph are as follows:
\begin{enumerate}
    \item $\{b,g,e,f,d,c,b\}$, $\{b,g,c\}$, $\{d,e,f\}$
    \item $\{h,j,i\}$
    \item $\{j,k,l\}$
\end{enumerate}
We will first show that each of these cycle is cop dominated.  Recall that a cycle is cop-dominated if the robber cannot avoid capture by entering that cycle and moving around it for the rest of the game.  We prove this for each cycle as follows:
\begin{enumerate}
    \item If the robber start from any vertex in $\{b,g,e,f,d,c,b\}$, it will be captured after one move since the cop begins at $a$.
    \item For the cycle $\{h,j,i\}$, we consider three cases:
    \begin{enumerate}
        \item If the robber starts at $h$, the cop will move to $g$, in the next move robber has to move to $j$, and so the cop moves to $e$. Now the robber has to leave the cycle $\{h,j,i\}$, for if the robber moves to $i$, it will get captured in the next step. Hence, the robber has been forced out of the cycle $\{h,j,i\}$.
        \item This case can be reduced to case 2(a) as follows: if the robber starts at $i$, the cop will move to $b$, in the next move robber has to move to $h$, and so the cop moves to $g$.
        \item This case can also be reduced to case 2(a) as follows: if the robber starts at $j$, the cop will move to $c$, in which case the robber has to move to $i$, and the cop moves to $b$. Now the robber and cop are in the same situation as mentioned in above case.
    \end{enumerate} 
    \item For the cycle $\{j,k,l\}$, we use an argument symmetric to that in Case 2.
\end{enumerate}
We now show that the cop can force the robber to move to either the cycle $\{h,j,i\}$ or the cycle $\{j,k,l\}$.  If the cop wants the robber to move to the cycle $\{h,j,i\}$, she moves in the cycle $\{b,g,c\}$. If she wants to force the robber to the cycle $\{j,k,l\}$ from the cycle $\{h,j,i\}$, she will move along the path $\{c,b,g,e\}$ (it is easy to check that, if the robber wishes to stay in the cycle $\{h,j,i\}$, then when the cop arrive at $e$ the robber will be on $j$).

Finally, we consider the cop number.  It is not hard to find a strategy so that $2$ cops can win, so must only show that the ograph is not cop-win.  First note that a single cop positioned on a vertex in $\{h,i,j,k,l\}$ cannot capture the robber, and thus capture must happen from $\{b,g,e,f,d,c\}$.  However, if the cop were to stay on the cycle $\{b,g,e,f,d,c\}$, then the robber is able move out of the cops' position's neighbourhood while staying in $\{h,i,j,k,l\}$.
\end{proof}

\section{Planar oriented graphs}\label{planar}

One of the first fundamental results in the study of the cop number parameter is Aigner and Fromme's proof that $\cop(G) \leq 3$ for any planar graph $G$ \cite{AF84}.  It has been recently shown that this bound does not hold for strong orientations of planar graphs.

\begin{thm}\cite{LO17}
There exists a strongly connected planar ograph $D$ for which $\cop(D) \geq 4$.
\end{thm}

On the other hand, it is known that planar ographs satisfy the Meyniel bound:

\begin{thm}\cite{LO17}
 If $D$ is a strongly connected orientation of a planar graph $G$, then $\cop(D) \in O(\sqrt{|V(D)|})$.
\end{thm}

The following question remains open:

\begin{ques}\label{planarograph}
Does there exist a constant $k$ such that $\cop(D) \leq k$ for every strongly connected planar ograph $D$?
\end{ques}

We are able to affirm Question \ref{planarograph} in the case of outerplanar ographs, showing that $\cop(D)=2$ for every strongly connected outerplanar ograph $D$ (recall that any digraph with no source must necessarily have $\cop(D) > 1$).

\begin{lem}\label{Lemmaouterplanar1}
An outerplanar ograph is Hamiltonian if and only if the exterior face is bounded by a Hamiltonian cycle.
\end{lem}

\begin{proof}
This follows immediately from a result due to \cite{S79} which states that any $2$-connected outerplanar graph has a unique Hamiltonian cycle.
\end{proof}

\begin{lem}\label{Lemmaouterplanar2}
An orientation $D$ of an outerplanar graph $G$ is strongly connected if and only if there exists a collection of Hamiltonian induced subdigraphs $\mathcal{H} = \{H_1,\ldots,H_k\}$ such that 
	\begin{compactitem}
    \item for all $1 \leq i < j \leq k$, $H_i \cap H_j$ is either an arc of $D$, a node of $D$, or empty;
    \item every arc on the outer face of $D$ is in a Hamiltonian cycle of some $H_i$.
    \item the graph $G_{\mathcal{H}}$ having $V(G_{\mathcal{H}}) = \mathcal{H}$ and $E(G_{\mathcal{H}}) = \{(H_i,H_j) : V(H_i) \cap V(H_j) \neq \emptyset\} \subseteq \mathcal{H}^2$ is a tree.
    \end{compactitem}
\end{lem}

\begin{proof}
The backward implication follows from the fact that that $G_{\mathcal{H}}$ is a tree and each $H_i$ is Hamiltonian.

Assume that $G$ is $2$-connected, noting that one may apply the argument below to any $2$-connected block of $G$.  We proceed by induction on $|V(D)|$, noting that the statement is trivially true when $D$ is an oriented cycle.  Consider an embedding of $G$ in the plane such that the outer face is a cycle $v_1v_2\cdots v_n$.  Let $e = v_iv_j$ be a chord such that every vertex in $\{v_{i+1},\ldots,v_{j-1}\}$ has degree $2$ in $G$.  In $D$, consecutive arcs in the path $v_iv_{i+1}\cdots v_j$ along $C$ must be oriented head to tail since $D$ is strongly connected.  Without loss of generality, assume that these arcs are $(v_i,v_{i+1}),(v_{i+1},v_{i+2}),\ldots,(v_{j-1}v_j)$.  By induction, $D' = D \setminus \{v_{i+1},\ldots,v_{j-1}\}$ contains a decomposition $\mathcal{H}'$ as described in the statement of the theorem.  If $(v_j,v_i) \in A(D)$, then we add the Hamiltonian cycle induced by $\{v_{i},\ldots,v_{j}\}$ to the collection, which will be a leaf in $G_\mathcal{H}$.  Otherwise, suppose $(v_i,v_j) \in A(D)$.  This arc lies on the outer face of $D'$, and thus lies in a Hamiltonian cycle of some $H_i$.  We then replace $H_i$ with $H_i \cup \{v_{i},\ldots,v_{j}\}$ in $\mathcal{H}'$ to obtain $\mathcal{H}$.  Since $G_{\mathcal{H}}$ is isomorphic to $G_{\mathcal{H'}}$, the proof is complete.
\end{proof}


\begin{lem}\label{Lemmaouterplanar3}
If $D$ is a Hamiltonian outerplanar ograph, then $\cop(D) = 2$.  
\end{lem}

\begin{proof}
Let $V(D) = \{v_1, \ldots, v_n\}$, with labels chosen so that $C = v_1v_2\cdots v_nv_1$ is a directed Hamiltonian cycle of $D$.  Let $G$ denote the underlying undirected graph of $D$.  We prove that the cops have a strategy so that, after each move, a cop occupies the end of a chord of $G$.  Call this cop the "stationary cop".  A chord $v_iv_j$ naturally partitions $V(D) \setminus \{v_iv_j\}$ according to the two directed paths left when deleting $v_i$ and $v_j$ from $C$.  The part of $V(D)$ in which the robber is located will be called the ``robber territory''.  We show that, after a finite number of steps, either the robber is captured or the non-stationary cop will occupy a new chord which shrinks the robber territory.  This cop then becomes the stationary cop, and the argument repeats.  Since the robber territory cannot shrink indefinitely, this method provides a capture strategy for two cops.  Note that, throughout, the subscript of $v_i$ will be taken modulo $n$.

Suppose that the stationary cop occupies an end of the chord $v_iv_j$ and suppose, without loss of generality, that $\{v_{i+1}, \ldots, v_{j-1}\}$ is the robber territory.  Let $k \in \{i+1, \ldots, j-1\}$ be the smallest index (modulo $n$) such that $v_k$ has total degree at least $3$ (note that, if none exists, the non-stationary cop has a clear winning strategy by walking along the induced directed path $v_iv_{i+1}\cdots v_{j-1}v_j$).  Let $v_kv_m$ be the chord of $G$ such that $v_m$ is the closest vertex in $\{v_{k+1}, \ldots, v_{j-1}\}$ to $v_j$.  The stationary cop begins by moving along each vertex of the directed path $v_iv_{i+1}\cdots v_m$ while the stationary cop remains in place.  If the robber has not been captured, then he now is on some vertex of $v_mv_{m+1}\cdots v_{j-1}v_j$.

If $v_k$ is the head of the directed arc $(v_k,v_m)$ and the robber's position is on the path $v_kv_{k+1}\cdots v_{m-1}v_m$, then the non-stationary cop becomes the stationary cop, and the robber's territory has decreased, as desired.  If $v_k$ is the tail of the directed arc $(v_m,v_k)$, then the non-active cop continues walking along $C$ until reaching $v_m$.  Again, if the robber's position is on the path $v_kv_{k+1}\cdots v_{m-1}v_m$, then the non-stationary cop becomes the stationary cop, and the robber's territory has decreased.

In each case above, there is the possibility that the robber's position is on the directed path $v_{m+1}\cdots v_{j-1}v_j$.  If $v_k$ is the head of the directed arc $(v_k,v_m)$, then the non-stationary cop moves from $v_k$ to $v_m$ (and if $v_k$ was the tail, then the non-stationary cop may already be considered to be on $v_m$).  Note that it is now impossible for the robber to return to $\{v_{i+1}, \ldots, v_{m-1}\}$, as no vertex in $\{v_{i+1}, \ldots, v_{k-1}\}$ is incident to a chord of $G$, no vertex in $\{v_{m+1},\ldots,v_{j}$ is incident to $v_k$ by choice of $m$, and no chord may cross $v_kv_m$.  Now the non-stationary cop continues moving along $C$ to the next vertex of total degree at least $3$ and repeats the same strategy as above.  At some point, the non-stationary cop will be able to become the stationary cop and reduce the robber territory, or will exhaust the vertices of total degree at least $3$ along $v_iv_{i+1}\cdots v_{j}$.  In the latter case, all that remains is an induced path along $C$ which terminates at $v_j$, and the robber is easily captured.
\end{proof}

\begin{thm}\label{outerplanar2}
If $D$ is a strongly connected outerplanar ograph, then $\cop(D) = 2$.
\end{thm}

\begin{proof}
Let $\mathcal{H}$ be the collection of Hamiltonian induced subdigraphs guaranteed by Lemma \ref{Lemmaouterplanar2}, and $G_{\mathcal{H}}$ its associated tree.  At each step, the cops will reside in the same $H_i$, and play a simple modification of the strategy given in the proof of Lemma \ref{Lemmaouterplanar3}.  If the robber is in $H_i$, then the cops play the strategy as stated.  Suppose the robber is in $H_r$, $r \neq i$.  Since $G_{\mathcal{H}}$ is a tree, there is a unique path in $G_{\mathcal{H}}$ from $H_i$ to $H_r$; let $H_j$ be the neighbour of $H_i$ on this path.  By Lemma \ref{Lemmaouterplanar2}, $H_i$ intersects $H_j$ either at a single node $v_i$ or along an arc $(v_i,v_j)$.  In either case, the cops then play the strategy from the proof of Lemma \ref{Lemmaouterplanar3} imagining $v_i$ as the robber's position.  Clearly the cops will either catch the robber in $H_i$ or will occupy $v_i \in H_j$.  In the latter case, the cops have reduced the robber's territory, and can then iterate the strategy again on $H_j$.  Eventually the robber will be captured or the cops will have ``chased'' the robber to $H_l$ where $H_l$ is a leaf of $G_\mathcal{H}$.  In this case, the strategy from the proof of Lemma \ref{Lemmaouterplanar3} will capture the robber.
\end{proof}

\section{Line Digraphs}\label{linedigraphs}

In \cite{DGP14}, cops and robbers played on the edge set of a graph is considered.  In that variation, an agent may move from one edge to any incident edge.  The fewest number of cops needed to capture the robber when playing on the edges of $G$ is denoted $\ecop(G)$, and it is proven that $\lceil\frac{\cop(G)}{2}\rceil \leq \ecop(G) \leq \cop(G)+1$.  Note that $\ecop(G) = \cop(L(G))$ where $L(G)$ denotes the line graph of $G$.

The line digraph of a digraph $D$, denoted $L(D)$ is defined similarly.  The nodes of $L(D)$ are the arcs of $D$, and there is an arc in $L(D)$ between $(a,b)$ and $(c,d)$ if and only if $b=c$.  We consider cops robbers played on the arcs of digraphs, where an agent may move from $(a,b)$ to $(c,d)$ only when $b=c$.  In other words, one plays the vertex version of cops and robbers on $L(D)$.  Somewhat surprisingly, we can show that the cop number of a digraph $D$ is equal to that of its line digraph if $D$ is strongly connected. 

\begin{prop}
If $D$ is strongly connected, then so is $L(D)$.
\end{prop}

\begin{thm}\label{linedigraph}
If $D$ is a strongly connected digraph $D$, then $\ecop(D) = \cop(D)$.
\end{thm}

\begin{proof}
First, suppose that $k$ cops have a winning strategy on $D$.  Since $D$ is strongly connected, we know that $k \geq 2$.  If the $k$ cops begin at vertices $v_1, \ldots, v_k$ in the regular game (repetition is allowed), then assign cop $C_i$ to an arc whose head is $v_i$ (such an arc exists because $D$ is strongly connected).  The robber then chooses a starting arc. The cops will play a strategy which mimics the regular game.  More precisely, if a cop occupies an arc $(a,b)$, then it imagines itself being on $b$ in the regular game.  If the next move in the regular game is to proceed to the node $c$, then the cop moves to the arc $(b,c)$ in the arc-version.  If the next move is to pass, then the cop passes.  When the robber moves from arc $(u,v)$ to arc $(v,w)$, then cops imagine that it has moved from $v$ to $w$ in the regular game.  It follows that at least one cop can eventually occupy an arc whose head is the same as the robber's arc.  If this is the same arc, then the game is finished.  If it is not, then any move by the robber will move it to an edge whose tail is incident to the head of the capturing cop's position.  If the robber stays put, then one of the remaining cops (recall that $\cop(D) \geq 2$) may move to occupy the robber's edge since $D$ is strongly connected.  Thus $k$ cops can catch the robber in the arc-version of the game.

Now, assume that $k$ cops have a winning strategy in the arc-version of the game.  If cop $C_i$ were to occupy the arc $(a_i,b_i)$ to start the arc-version, then $C_i$ occupies $b_i$ to start the node version.  If the robber chooses to start at $v$, then the cops imagine the robber occupies any arc whose head is $v$.  The cops then play a strategy where, instead of moving from $(a_i,b_i)$ to $(b_i,c_i)$ in the arc version, they move from $b_i$ to $c_i$ in the node version, and the robber is imagined to occupy the arc consisting of its current position as its head and its previous position as its tail.  When a cop catches the robber's arc in the imagined arc version of the game, she will by definition occupy the same node as the robber (the head of the arc).
\end{proof}

Note that the condition of being strongly connected cannot be dropped.  If $D$ is a star with all arcs oriented toward the leaves, then $\cop(D) = 1$ while $\ecop(D) = |V(D)|-1$ (one cop is needed on each arc).

We note that one could also prove Theorem \ref{linedigraph} using structural characterizations of line digraphs.  In \cite{LW98}, a coreflexive set (henceforth coreset) of vertices in a digraph $D$ is defined to be a minimal set $X \subset V(D)$ such that $N^-\left(N^+(X)\right) = X$.  It is shown that the collection of all coresets of $D$ partitions $V(D)$.  The authors then prove a number of equivalent characterization of line digraphs (see \cite{LW98}, Theorem 6), which imply such classic characterizations as Heuchenne's condition and the Geller-Harary condition.  One of these is as follows:

\begin{thm}\cite{LW98}
Let $D$ be a digraph, and let $U_1, \ldots U_k$ be the partition of $V(D)$ into coresets.  If, for each $i \in [k]$, the subdigraph induced by $U_i \cup N^+(U_i)$ contains every possible arc from $U_i$ to $N^+(U_i)$ and no others, then $D$ is a line digraph.
\end{thm}

It is easy to see that, if $D$ is a strongly connected digraph, then $L(D)$ has $|V(D)|$ coresets, each of which is a stable set and each of which corresponds to the set of incoming arcs for some vertex.  Further, if one contracts each coreset in $L(D)$, the result is a digraph isomorphic to $D$ (each contracted coreset becomes the vertex to which its arc point in $D$).  Thus, any strategy one plays on $D$ can be translated to the coresets of $L(D)$ and (since coresets are stable sets in $L(D)$) the converse holds as well.

A natural problem arises from this ability to move between line digraphs and digraphs obtained by contracting coresets -- is there any relation between the cop number of an arbitrary digraph and that of the digraph obtained by contracting its coresets and deleting parallel arcs?  Unfortunately, the answer appears to be ``no'', in general.

\begin{prop}\label{contract}
For arbitrarily large $n$, there exists an $n$-vertex strongly connected ograph $D$ for which $\cop(D) \in \Om(\sqrt{n})$ whose coreset contraction has $Y$ has $\cop(Y)=2$.
\end{prop}

\begin{proof}
Let $D$ be the digraph with vertex set $A_1 \cup A_2 \cup B_1 \cup B_2$ such that $D[A_i \cup B_j]$ is the point-line incidence graph of the projective plan of order $q$, where $A_i$ is the set of points and $B_j$ is the set of lines.  For each edge $(a,b) \in A_i \times B_i$, we orient the arc as $(a,b)$.  For each edge $(a,b) \in A_i \times B_{i+1}$ (indices taken modulo $2$), we orient the edge $(b,a)$.  It is easy to see that this graph is strongly connected and has $d^-(v) = d^+(v) = q+1$ for every $v \in V(D)$.  Since the undirected girth is $6$, $\cop(D) \geq d^+(v) = \frac{q-1}{2}$ by Lemma \ref{dirlowbound}.  Since $|V(D)| = 4(q^2+q+1)$, we have our desired $D$.  It is then not hard to see that of $A_1, A_2, B_1, B_2$ is a coreset partition, $Y$ is a directed cycle of length $4$, and so $\cop(Y) = 2$.
\end{proof}

While this proposition is illustrative of how different the two cop numbers may be, it is also the special case of a much more general result, which states that if one iterates this procedure, one \textit{always} obtains a digraph with low cop number.  Let $D_0$ be an arbitrary digraph and define the sequence $\{D_n\}_{n=0}^{\infty}$ where $D_{i}$ ($i \geq 1$) is obtained by contracting all coresets of $D_{i-1}$; call $\{D_n\}_{n=0}^{\infty}$ the coreset contraction sequence for $D_0$.  

\begin{thm}\cite{LW98}
For any digraph $D_0$, the coreset contraction sequence $\{D_n\}_{n=0}^{\infty}$ converges either to a directed path or to a directed cycle with a directed path (possibly of length $0$) leaving some vertex of the cycle.
\end{thm}

\begin{cor}
Given a digraph $D_0$ with coreset contraction sequence $\{D_n\}_{n=0}^{\infty}$, the sequence $\{\cop(D_n)\}_{n=0}^{\infty}$ converges to $1$ or $2$.
\end{cor}



\section{Optimal Strategies}\label{strategy}

In this section, we consider two aspects of game play when cops and robber all play optimally.  First, we consider the number of rounds required for $k$ cops to capture a robber in a $k$-cop win ograph.  Second, we consider some questions regarding the relative positions of the cops and robber which have been examined for undirected graphs but not for ographs (or directed graphs in general).

\subsection{Capture time}

We first note that the polynomial-time algorithm given in \cite{HM06} which determines whether or not a digraph has cop number at most $k$ also gives the minimum number of rounds in which they can win.  Thus, if $\cop(D)$ is known, then determining $\capt(D)$ is in $\mathbf{P}$.

As mentioned in the introduction, Kinnersley \cite{K18} showed that for every $k\geq 1$ there exists an $n$-vertex directed graph for which $\cop(D) =k$ and $\capt(D) \in \Om((\tfrac{n}{k})^{k+1})$.  There are two things notable about this result.  The first is that it matches the natural upper bound, as there are on the order of $n^{k+1}$ distinct combinations of where $k$ cops and a single robber may be located.  The second is that, unlike in the case of undirected graphs where cop-win graphs have capture time at most linear in $n$, there exist $n$-vertex directed graphs for which the capture time is in $\Th(n^2)$.  We present an alternate construction, first given in the thesis of one this paper's co-authors \cite{Kom13}, of an infinite family of cop-win digraphs with quadratic capture time which does not require the impressive machinery of \cite{K18}.

\begin{thm} 
\label{ring digraph}
There exists an infinite family of directed graphs on $n$ vertices with capture time $\Th(n^2)$.
\end{thm}

\begin{proof}


We define a \textit{ring digraph} $R(k)$ to be a reflexive directed graph on $n = 2k+1$ vertices consisting of:
\begin{itemize}
\item an ``outer ring'' comprised of a (counterclockwise)-directed $k$-cycle
\item an ``inner ring'' comprised of a (counterclockwise)-directed $(k-1)$-cycle 
\item arcs from a vertex in the inner ring to a vertex in the outer ring configured such that $k-2$ vertices in the inner ring are incident with one such arc, and 1 vertex in the inner ring is incident with two such arcs
\item an ``internal vertex'' (marked C in Figure~\ref{ringdigraph}) that is out-directed to every vertex in the inner ring
\item an ``external vertex'' (marked R in Figure~\ref{ringdigraph}) incident with two arcs as in Figure~\ref{ringdigraph}
\end{itemize}

\begin{figure}[h!]
\centering
\includegraphics[scale=0.45]{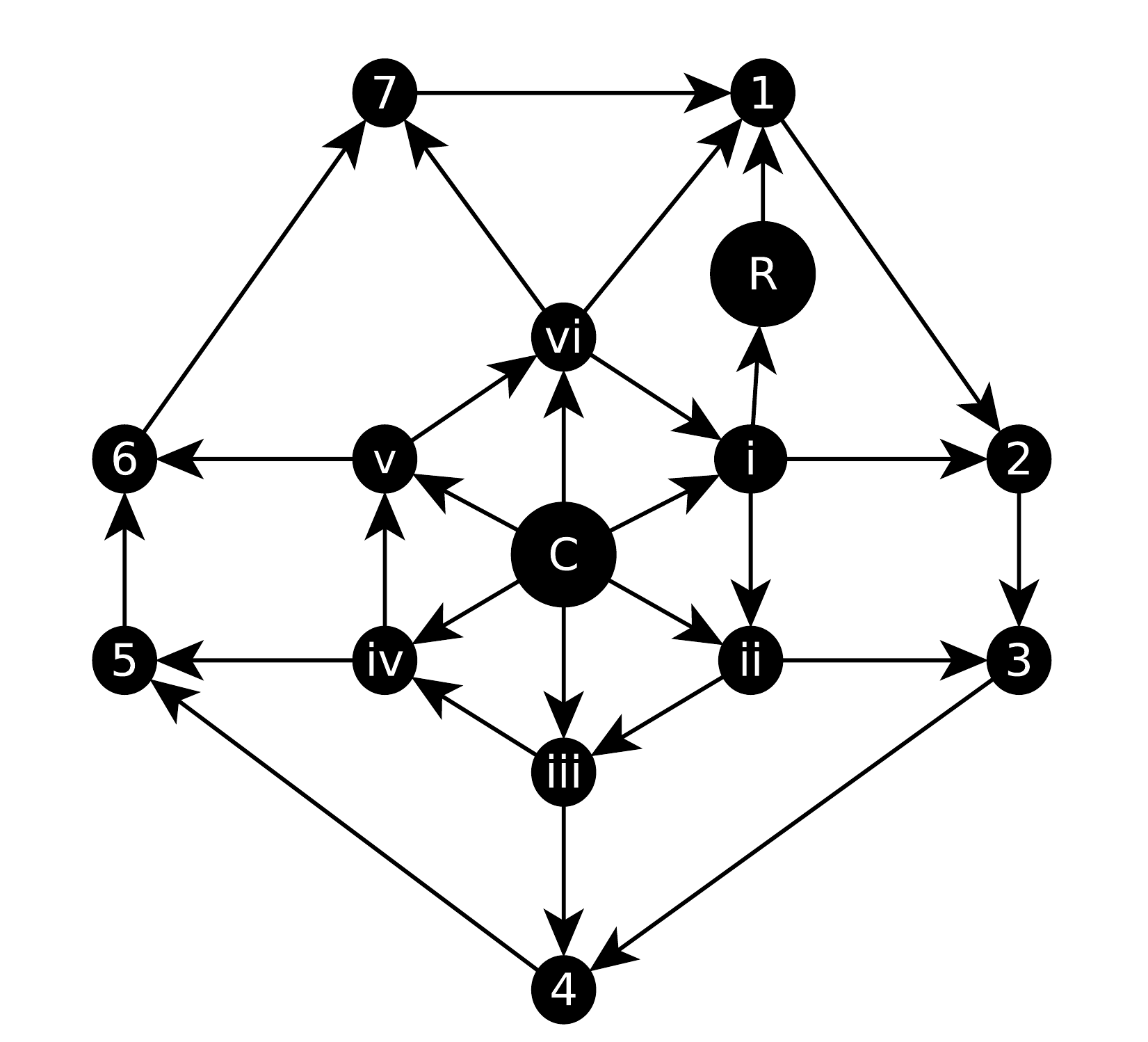}
\caption{The ring digraph $R(7)$}
\label{ringdigraph}
\end{figure}

We claim that for all $k$, $R(k)$ is cop-win and the capture time on $R(k)$ is $\Th(n^2)$.  In particular, the capture time on $R(k)$ is $(n^2-4n+3)/4$, where $n = |V(G)|$. 

We describe a strategy where the cop starts on the vertex labeled $C$ in Figure~\ref{ringdigraph} and show that with both players moving optimally, the game takes $(k-1)^2+1 = k^2 - 2k$ moves from this starting position. Then we show that this is the best starting position for the cop.

Let $C_i$ be the cop's position after the $i^{th}$ move (so $C_0$ is her starting position), and define $R_i$ analogously. Then if $C_0$ = $C$ and $R_0 = R$, then the cop must play $C_1 = i$ (the robber can remain on $R$ until the cop is on this vertex). The robber's only choice is to play $R_1 = 1$. Now the cop must move along the inner ring, as the robber can pass whenever the cop passes (and if she moves onto the outer ring, the robber wins). Now the only way that the cop can win is to be on vertex $vi$ in Figure~\ref{ringdigraph} (that is, the vertex on the inner ring with outdegree 3) while the robber is on vertex 7. That is, the cop wants to be one step behind the robber. Then the robber can either pass or move onto vertex 1, and will be captured in one move in either case. But after move 1, the cop is one step ahead of (i.e. $k{-}1$ steps behind) the robber. Therefore it will take $k{-}1$ full trips (i.e. $(k{-}1)^2$ moves) to come to this position, as every time the cop makes a full trip around the inner ring, she is one more step ahead of the robber.

If the cop's initial position was anywhere but on vertex $C$, then the robber could play $R_0 = C$ and remain there for the rest of the game, so $C_0 = C$ is the cop's only option.
 \end{proof}

Interestingly, the problem of settling the capture time in tournaments remains open:

\begin{ques}[Sl\'ivlov\'a \cite{S15}]
Does there exist a constant $t$ such that, for every tournament $T$, $\capt(T) \leq t$? 
\end{ques}

\subsection{Relative positions of cops and robber}

For clarity, we remind the reader that ``optimal'' play means that the cop attempts to win in as few rounds as possible while the robber attempts to prolong his capture as much as possible.
Two interesting conjectures related to the positions of  the game's agents in the undirected game were posed by Hahn, first via private communication to the research community and then published in \cite{B+13}:

\begin{conj}\cite{B+13}
If $G$ is a cop-win (undirected) graph, then the distance between the cop and robber does not increase in consecutive rounds in optimal play.
\end{conj}

\begin{conj}\cite{B+13}
If $G$ is a cop-win (undirected) graph, then the cop visits each vertex at most once in optimal play.
\end{conj}

These conjectures were each refuted in \cite{B+13}.  We show that analogous versions of these conjectures would not hold in the context of ographs.

\begin{thm}
There exist cop-win oriented graphs in which, if both cop and robber play optimally, the distance between the cop and robber must strictly increase in two consecutive rounds.
\end{thm}



\begin{proof}
Consider the graph in Figure \ref{distanceex}, where $n$ is taken to be a sufficiently large positive integer.  The vertex $a$, whose outneighbourhood is a $\{v_i \,\mid\, 1 \leq i \leq 12\}$ (not all arcs are drawn), is a source and hence the starting point for a single cop.  The vertices $w_1,w_2,...,w_{n-1},w_{n}$ are oriented from $w_i$ to $w_{i+1}$ and the vertices $u_1,u_2,...,u_{n-1},u_{n}$ are oriented from $u_i$ to $u_{i-1}$.
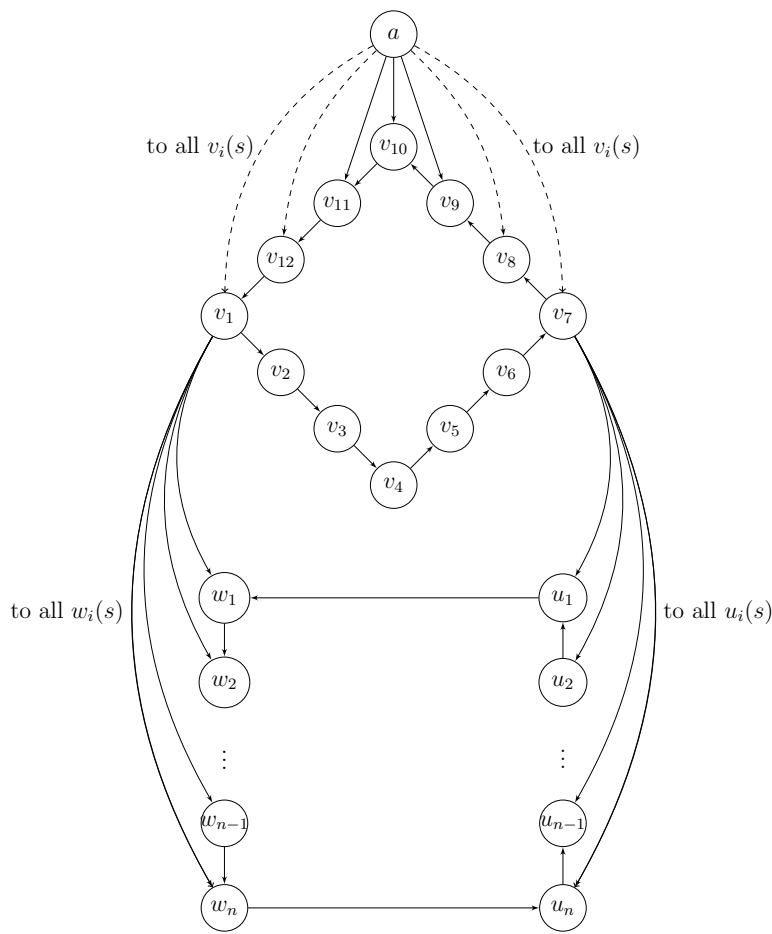
\begin{figure}[h!]
\begin{center}
\scalebox{0.75}{
\begin{tikzpicture}

\tikzset{vertex/.style = {shape=circle,draw,minimum size=2 em}}
\tikzset{edge/.style = {->,> = latex'}}

\node[vertex] (a) at  (3,5) {$a$};

\node[vertex, label=center:$v_{1}$] (m) at (0,0) {};
\node[vertex, label=center:$v_{12}$] (l) at (1,1) {};
\node[vertex, label=center:$v_{11}$] (k) at (2,2) {};
\node[vertex, label=center:$v_{10}$] (j) at (3,3) {};
\node[vertex, label=center:$v_{9}$] (i) at (4,2) {};
\node[vertex, label=center:$v_{8}$] (h) at (5,1) {};
\node[vertex, label=center:$v_{7}$] (g) at (6,0) {};
\node[vertex, label=center:$v_{6}$] (f) at (5,-1) {};
\node[vertex, label=center:$v_{5}$] (e) at (4,-2) {};
\node[vertex, label=center:$v_{4}$] (d) at (3,-3) {};
\node[vertex, label=center:$v_{3}$] (c) at (2,-2) {};
\node[vertex, label=center:$v_{2}$] (b) at (1,-1) {};

\node[vertex] (n) at (6,-5) {$u_1$};
\node[vertex] (o) at (6,-6.5) {$u_2$};
\node[vertex, label=center:$u_{n-1}$] (p) at (6,-9) {};
\node[vertex, label=center:$u_{n}$] (q) at (6,-10.5) {};

\node[vertex] (r) at (0,-5) {$w_1$};
\node[vertex] (s) at (0,-6.5) {$w_2$};
\node[vertex, label=center:$w_{n-1}$] (t) at (0,-9) {};
\node[vertex, label=center:$w_{n}$] (u) at (0,-10.5) {};


\draw[edge] (b) to (c);
\draw[edge] (c) to (d);
\draw[edge] (d) to (e);
\draw[edge] (e) to (f);
\draw[edge] (f) to (g);
\draw[edge] (g) to (h);
\draw[edge] (h) to (i);
\draw[edge] (i) to (j);
\draw[edge] (j) to (k);
\draw[edge] (k) to (l);
\draw[edge] (l) to (m);
\draw[edge] (m) to (b);

\draw[edge] (o) to (n);
\path (p) to node {\vdots} (o);
\draw[edge] (q) to (p);

\draw[edge] (r) to (s);
\path (s) to node {\vdots} (t);
\draw[edge] (t) to (u);

\draw[edge] (n) to (r);
\draw[edge] (u) to (q);

\draw[edge] (m)  to[bend right] (r);
\draw[edge] (m)  to[bend right] (s);
\draw[edge] (m)  to[bend right] (t);
\draw[edge] (m)  to[bend right] (u);

\draw[edge] (g)  to[bend left] (n);
\draw[edge] (g)  to[bend left] (o);
\draw[edge] (g)  to[bend left] (p);
\draw[edge] (g)  to[bend left] (q);

\draw[edge] (a)  to (j);
\draw[edge] (a)  to (i);
\draw[edge, dashed] (a)  to[bend left=20] (h);

\draw[edge] (a)  to (k);
\draw[edge, dashed] (a)  to[bend right=20] (l);

\path[->]
(a) edge [bend right=-30, dashed] node[right] {to all $v_i(s)$} (g)
(a) edge [bend left=-30, dashed] node[left] {to all $v_i(s)$} (m)
(g) edge [bend left] node[right] {to all $u_i(s)$} (q)
(m) edge [bend right] node[left] {to all $w_i(s)$} (u);

\end{tikzpicture}
}
\caption{A cop-win graph where the cop must increase her distance from the robber}\label{distanceex}
\end{center}
\end{figure}

In optimal play, the robber must clearly begin on some $u_i$ or $w_i$.  If the robber starts from any vertex on $w_i$, $i \neq n$, or any $u_i$, $i \neq 1$, then it will clearly be captured in 2 steps. 
Hence we have that the robber needs to start from either $w_n$ or $u_1$ to play optimally. As these two cases are symmetric, assume the robber begins at $w_n$.
In the first round, the cop moves from $a$ to $v_1$ and the robber moves to $u_n$ to avoid capture.  Note that the distance from the cop to the robber is $2$ at this point in the game, attained by the path $v_1,w_n,u_n$.  However, the cop will not move along this path, since entering the cycle induced by $\{u_1,\ldots,u_n,w_1,\ldots,w_n\}$ cannot lead to capture in optimal play (if the robber resides on a cycle, then the cop must capture him from outside the cycle).  Thus, the cop's optimal strategy is to move along the path $v_1,v_2,v_3,v_4,v_5,v_6,v_7$.  After these seven rounds have completed, the robber will occupy some $u_i$, which is in the outneighbourhood of $v_7$, and thus the cop wins on her next turn.  In this strategy, it is easy to see from the second to third rounds of the game, the distance from the cop to the robber changed from $2$ to $6$.

Finally, we note that this particular example can easily be generalize so that the increase in distance can be made arbitrarily large.  To do this, we change the cycle of length $12$ in the outneighbourhood of $a$ to be one of arbitrary even length with antipodal vertices taking the roles of $v_1$ and $v_7$ and then making $n$ sufficiently large.
\end{proof}

Finally, we show that a cop may have to revisit a vertex after having left it during optimal play.
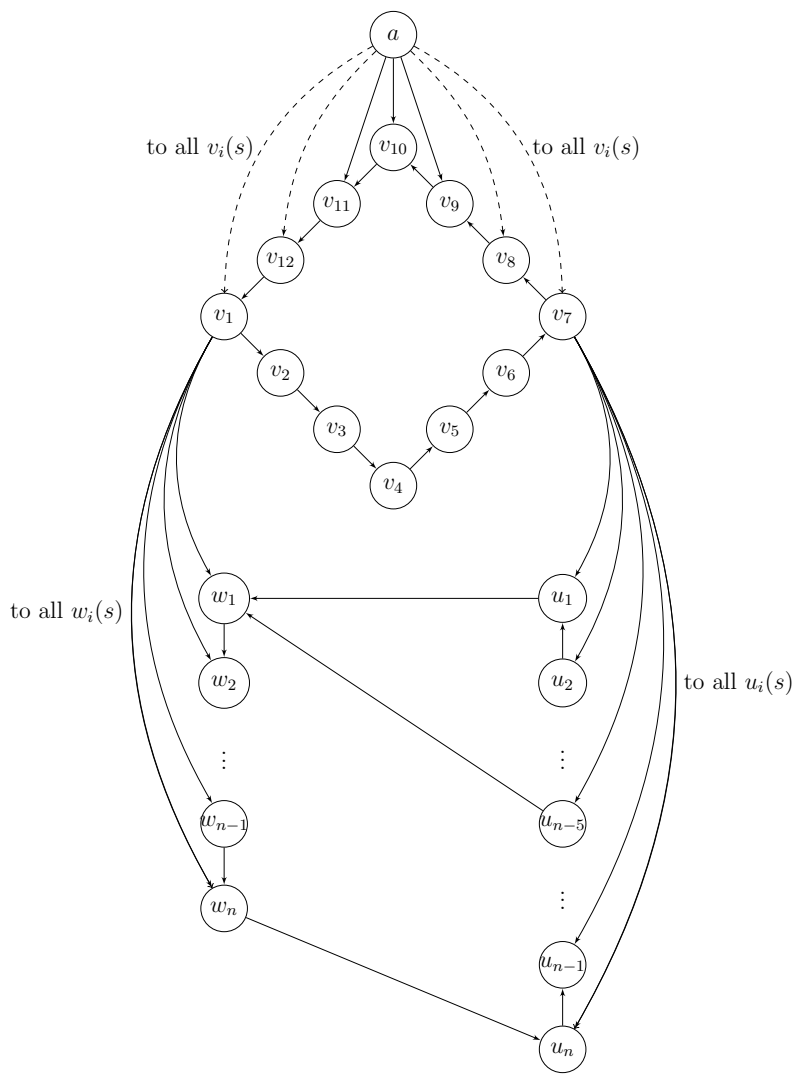
\begin{figure}[h!]
\begin{center}
\scalebox{0.75}{
\begin{tikzpicture}

\tikzset{vertex/.style = {shape=circle,draw,minimum size=2 em}}
\tikzset{edge/.style = {->,> = latex'}}

\node[vertex] (a) at  (3,5) {$a$};

\node[vertex, label=center:$v_{1}$] (m) at (0,0) {};
\node[vertex, label=center:$v_{12}$] (l) at (1,1) {};
\node[vertex, label=center:$v_{11}$] (k) at (2,2) {};
\node[vertex, label=center:$v_{10}$] (j) at (3,3) {};
\node[vertex, label=center:$v_{9}$] (i) at (4,2) {};
\node[vertex, label=center:$v_{8}$] (h) at (5,1) {};
\node[vertex, label=center:$v_{7}$] (g) at (6,0) {};
\node[vertex, label=center:$v_{6}$] (f) at (5,-1) {};
\node[vertex, label=center:$v_{5}$] (e) at (4,-2) {};
\node[vertex, label=center:$v_{4}$] (d) at (3,-3) {};
\node[vertex, label=center:$v_{3}$] (c) at (2,-2) {};
\node[vertex, label=center:$v_{2}$] (b) at (1,-1) {};

\node[vertex] (n) at (6,-5) {$u_1$};
\node[vertex] (o) at (6,-6.5) {$u_2$};
\node[vertex, label=center:$u_{n-5}$] (pp) at (6,-9) {};
\node[vertex, label=center:$u_{n-1}$] (p) at (6,-11.5) {};
\node[vertex, label=center:$u_{n}$] (q) at (6,-13) {};

\node[vertex] (r) at (0,-5) {$w_1$};
\node[vertex] (s) at (0,-6.5) {$w_2$};
\node[vertex, label=center:$w_{n-1}$] (t) at (0,-9) {};
\node[vertex, label=center:$w_{n}$] (u) at (0,-10.5) {};


\draw[edge] (b) to (c);
\draw[edge] (c) to (d);
\draw[edge] (d) to (e);
\draw[edge] (e) to (f);
\draw[edge] (f) to (g);
\draw[edge] (g) to (h);
\draw[edge] (h) to (i);
\draw[edge] (i) to (j);
\draw[edge] (j) to (k);
\draw[edge] (k) to (l);
\draw[edge] (l) to (m);
\draw[edge] (m) to (b);

\draw[edge] (o) to (n);
\path (pp) to node {\vdots} (o);
\path (p) to node {\vdots} (pp);
\draw[edge] (q) to (p);

\draw[edge] (r) to (s);
\path (s) to node {\vdots} (t);
\draw[edge] (t) to (u);

\draw[edge] (n) to (r);
\draw[edge] (u) to (q);

\draw[edge] (m)  to[bend right] (r);
\draw[edge] (m)  to[bend right] (s);
\draw[edge] (m)  to[bend right] (t);
\draw[edge] (m)  to[bend right] (u);

\draw[edge] (g)  to[bend left] (n);
\draw[edge] (g)  to[bend left] (o);
\draw[edge] (g)  to[bend left] (p);
\draw[edge] (g)  to[bend left] (pp);
\draw[edge] (g)  to[bend left] (q);

\draw[edge] (a)  to (j);
\draw[edge] (a)  to (i);
\draw[edge, dashed] (a)  to[bend left=20] (h);

\draw[edge] (a)  to (k);
\draw[edge, dashed] (a)  to[bend right=20] (l);

\draw[edge] (pp) to (r);

\path[->]
(a) edge [bend right=-30, dashed] node[right] {to all $v_i(s)$} (g)
(a) edge [bend left=-30, dashed] node[left] {to all $v_i(s)$} (m)
(g) edge [bend left] node[right] {to all $u_i(s)$} (q)
(m) edge [bend right] node[left] {to all $w_i(s)$} (u);

\end{tikzpicture}
}
\caption{A cop-win graph where the cop must revisit a vertex in optimal play}\label{mustrevisit}
\end{center}
\end{figure}

\begin{thm}
There exists a cop win oriented graph in which cop must revisit a vertex in optimal play.
\end{thm}

\begin{proof}
Consider the graph in Figure \ref{mustrevisit}.  Note that this graph is obtained from the one in Figure \ref{distanceex} by adding the arc $(u_{n-5},w_1$ and leaving all remaining vertices and arcs unchanged.

The cop begins on $a$.  If the robber starts from any vertex in $v_i(s)$, the capture time will be one. If the robber begins on $w_i$ or $u_{i}$ the capture time will be two unless he begins on $w_n, u_1\;or\;u_{n-5}$. We proceed by cases.
\begin{enumerate}
    \item If the robber begins at $u_1$, then the cop will move to $v_7$, forcing the robber to move to $w_1$.  In the second through seventh rounds, the will move from $v_7$ to $v_1$ and the robber will finish on some $w_i$ (as in the proof of the previous theorem, $n$ will be chosen to be sufficiently large).  The cop then wins in the eighth round.
    \item If the robber starts from position $u_{n-5}$, then the cop moves to $v_7$, forcing the robber to move to $w_1$. This is similar to the above case, and the cop will again win after $8$ rounds.
    \item Let us assume the robber starts from $w_n$. In the first step, the cop moves from $a$ to $v_1$, forcing the robber to move to $u_n$. In the next five rounds, the cop will move from $v_1$ to $v_6$, and robber moves from $u_n$ to $u_{n-5}$. In the seventh step, the cop moves to $v_7$. The robber is forced to move to $w_1$. From above cases we know the cop requires seven more steps to capture the robber. Hence, the cop moves from $v_7$ again to $v_1$ in six more rounds (i.e. from $8^{th}$ to $13^{th}$ round), and then captures it in the $14^{th}$ round.
\end{enumerate}
The third case then represents optimal play, and the cop revisits vertex $v_1$ in this case.
\end{proof}

\section{Acknowledgements}
This work was supported in part by the Fonds de recherche du Qu\'ebec -- Nature et technologies (FRQNT) and the Natural Sciences and Engineering Research Council of Canada (NSERC).  Devvrit Khatri and Nithish Kumar also acknowledge the generous support received from MITACS and the Globalink Program.

\let\OLDthebibliography\thebibliography
\renewcommand\thebibliography[1]{
  \OLDthebibliography{#1}
  \setlength{\parskip}{0pt}
  \setlength{\itemsep}{0pt plus 0.3ex}
}

\bibliographystyle{siam}      
\small{
\bibliography{references.bib}   

\begin{thebibliography}{10}

\bibitem{AF84}
{\sc M.~Aigner and M.~Fromme}, {\em A game of cops and robbers}, Discrete
  Applied Mathematics, 8 (1984), pp.~1 -- 12.

\bibitem{BI93}
{\sc A.~Berarducci and B.~Intrigila}, {\em On the cop number of a graph},
  Advances in Applied Mathematics, 14 (1993), pp.~389 -- 403.

\bibitem{BC09}
{\sc A.~Bonato and E.~Chiniforooshan}, {\em Pursuit and evasion from a
  distance: Algorithms and bounds}, in Proceedings of the Meeting on Analytic
  Algorithmics and Combinatorics, ANALCO '09, Philadelphia, PA, USA, 2009,
  Society for Industrial and Applied Mathematics, pp.~1--10.

\bibitem{BN}
{\sc A.~Bonato and R.~J. Nowakowski}, {\em The game of cops and robbers on
  graphs}, vol.~61, American Mathematical Society Providence, 2011.

\bibitem{B+13}
{\sc M.~Boyer, S.~El~Harti, A.~El~Ouarari, R.~Ganian, G.~Hahn, C.~Moldenauer,
  I.~Rutter, B.~Th{\'e}riault, and M.~Vatshell}, {\em Cops and robbers: remarks
  and problems}, Journal of Combinatorial Mathematics and Combinatorial
  Computing, 85 (2013).

\bibitem{DGGH16}
{\sc E.~Darlington, C.~Gibbons, K.~Guy, and J.~Hauswald}, {\em Cops and robbers
  on oriented graphs}, RHIT Undergrad. Math. J., 17 (2016), pp.~202--209.

\bibitem{DGP14}
{\sc A.~Dudek, P.~Gordinowicz, and P.~Pra{\l}at}, {\em Cops and robbers playing
  on edges}, Journal of Combinatorics, 5 (2014), pp.~131--153.

\bibitem{FKL12}
{\sc A.~Frieze, M.~Krivelevich, and P.-S. Loh}, {\em Variations on cops and
  robbers}, Journal of Graph Theory, 69 (2012), pp.~383--402.

\bibitem{GKS18}
{\sc I.~Gromovikov, W.~B. Kinnersley, and B.~Seamone}, {\em Fully active cops
  and robbers}, arXiv preprint arXiv:1808.06734,  (2018).

\bibitem{HM06}
{\sc G.~Hahn and G.~MacGillivray}, {\em A note on k-cop, l-robber games on
  graphs}, Discrete Mathematics, 306 (2006), pp.~2492 -- 2497.
\newblock Creation and Recreation: A Tribute to the Memory of Claude Berge.

\bibitem{H87}
{\sc Y.~O. Hamidoune}, {\em On a pursuit game on cayley digraphs}, European
  Journal of Combinatorics, 8 (1987), pp.~289 -- 295.

\bibitem{H08}
{\sc A.~Hill}, {\em Cops and robbers: {T}heme and variations}, ProQuest LLC,
  Ann Arbor, MI, 2008.
\newblock Thesis (Ph.D.)--Dalhousie University (Canada).

\bibitem{Hthesis}
{\sc S.~A. Hosseini}, {\em Game of Cops and Robbers on Eulerian Digraphs}, PhD
  thesis, Simon Fraser University, 2018.

\bibitem{HM18}
{\sc S.~A. Hosseini and B.~Mohar}, {\em Game of cops and robbers in oriented
  quotients of the integer grid}, Discrete Math., 341 (2018), pp.~439--450.

\bibitem{K15}
{\sc W.~B. Kinnersley}, {\em Cops and robbers is {EXPTIME}-complete}, J.
  Combin. Theory Ser. B, 111 (2015), pp.~201--220.

\bibitem{K18}
\leavevmode\vrule height 2pt depth -1.6pt width 23pt, {\em Bounds on the length
  of a game of {C}ops and {R}obbers}, Discrete Math., 341 (2018),
  pp.~2508--2518.

\bibitem{Kom13}
{\sc N.~Komarov}, {\em Capture time in variants of cops \& robbers games}, PhD
  thesis, Dartmouth College, 7 2013.

\bibitem{LW98}
{\sc X.~Liu and D.~B. West}, {\em Line digraphs and coreflexive vertex sets},
  Discrete Mathematics, 188 (1998), pp.~269 -- 277.

\bibitem{LO17}
{\sc P.-S. Loh and S.~Oh}, {\em Cops and robbers on planar-directed graphs},
  Journal of Graph Theory,  (2017), pp.~n/a--n/a.

\bibitem{NW83}
{\sc R.~Nowakowski and P.~Winkler}, {\em Vertex-to-vertex pursuit in a graph},
  Discrete Mathematics, 43 (1983), pp.~235 -- 239.

\bibitem{Q83}
{\sc A.~Quilliot}, {\em Th\'ese d'Etat}, PhD thesis, Universit\'e de Paris VI,
  1983.

\bibitem{S15}
{\sc V.~Sl\'ivov\'a}, {\em Cops and robber game on directed complete graphs}.
\newblock 2015.

\bibitem{S79}
{\sc M.~M. Sys{\l}o}, {\em Characterizations of outerplanar graphs}, Discrete
  Mathematics, 26 (1979), pp.~47 -- 53.

\end{thebibliography}
}


\end{document}